\documentclass[10pt,oneside,leqno]{amsart}
\usepackage{amsmath,amsfonts,amssymb,amsthm,mathscinet}
\usepackage[mathcal]{euscript}
\usepackage{enumitem}
\setlist[enumerate]{label={\rm(\roman*)}}
\usepackage{color,comment,mathtools}
\usepackage[a4paper, left=2cm, right=2cm, top=3cm, bottom=3cm]{geometry}
\usepackage[numbers,sort&compress]{natbib}
\usepackage{hyperref,url,doi,xcolor}
\hypersetup{
	colorlinks,
	linkcolor=teal,
	urlcolor=magenta,
	citecolor=violet}
\usepackage[a-2u]{pdfx}

\renewcommand{\d}{{\mathrm d}}
\newcommand{\R}{\mathbb{R}}
\newcommand{\N}{\mathbb{N}}

\newcommand{\DD}{\mathcal{D}}

\newcommand{\MM}{\mathcal{M}}

\newcommand{\RR}{\mathcal{R}}

\let\tilde\widetilde
\newcommand{\dom}{(\R_+)}

\newtheorem{theorem}{Theorem}[section]
\newtheorem*{theorem*}{Theorem}
\newtheorem{lemma}[theorem]{Lemma}
\newtheorem{corollary}[theorem]{Corollary}

\newtheorem{remark}[theorem]{Remark}
\newtheorem{example}[theorem]{Example}
\expandafter\let\expandafter\oldproof\csname\string\proof\endcsname
\let\oldendproof\endproof
\renewenvironment{proof}[1][\proofname]{%
  \oldproof[\bf #1]%
}{\oldendproof}
\numberwithin{equation}{section}

\usepackage{xspace}
\makeatletter
\DeclareRobustCommand\onedot{\futurelet\@let@token\@onedot}
\def\@onedot{\ifx\@let@token.\else.\null\fi\xspace}
\def\eg{e.g\onedot} 
\def\ie{i.e\onedot} 
\def\ae{a.e\onedot} 
\def\ri{r.i\onedot} 
\makeatother

\makeatletter
\def\paragraph{\bigskip\@startsection{paragraph}{4}%
  \z@\z@{-\fontdimen2\font}%
  {\normalfont\bfseries}}
\makeatother

\newcommand{\fT}{f\chi_T}
\newcommand{\chn}{\chi_{n}}

\begin{document}

\title{Increasingly global convergence of Hermite series}

\begin{abstract}
We study the convergence of Hermite series of measurable functions on the real line.
We characterize the norm convergence of truncated partial Hermite sums in rearrangement invariant spaces
provided that the truncations increase sufficiently slowly.
Moreover, we provide necessary and sufficient conditions for convergence in the Orlicz modular.
\end{abstract}

\author[V.~Musil]{V\'\i t Musil\textsuperscript{1,*}}
\address{\textsuperscript{1}
Department of Computer Science,
Masaryk University,
Brno, Czech Republic
}
\email{musil@fi.muni.cz} 
\urladdr{%
	0000-0001-6083-227X 
}
\address{\textsuperscript{*}
Corresponding author
}

\author[S.~Spektor]{S. Spektor\textsuperscript{2}}
\address{\textsuperscript{2}
Quantitative Science Department,
Canisius University,
2001 Main Street, Buffalo, NY, USA
}
\email{spektors@canisius.edu}


\maketitle

\bibliographystyle{abbrv}

\section*{How to cite this paper}
\noindent
This paper has been accepted for publication in \emph{Journal of Mathematical Analysis and Applications} and is available at
\begin{center}
	\url{https://doi.org/10.1016/j.jmaa.2025.130360}.
\end{center}
Should you wish to cite this paper, the authors would like to cordially ask you
to cite it appropriately.

\section{Introduction and main result}
\label{S:intro}

\paragraph{Hermite series}

The $k$-th Hermite function, $h_k$, is given at $x\in\R$ by
\begin{equation*}
	h_k(x) 
		= (-1)^k \gamma_k e^{\frac{x^2}{2}} \frac{\d^k e^{-x^2}}{\d x^k}
		= H_k(x) e^{\frac{-x^2}{2}}
		\quad\text{for $k=0,1,\dots$},
\end{equation*}
in which $\gamma_k = \pi^{-1/4}2^{-k/2}(k!)^{-1/2}$ and $H_k$ is the
Hermite polynomial of degree $k$. Given a suitable function $f$ on $\R$,
its Hermite series is
\begin{equation*}
	\sum_{k=0}^{\infty} c_k h_k,
	\quad\text{where}\quad
	c_k = \int_{\R} f h_k
		\quad\text{for $k=0,1,\dots$}
\end{equation*}
We denote the $n$-th partial sum of the Hermite series of~$f$ by $S_nf= \sum_{k=0}^{n} c_k h_k$.

Hermite series, often under the name Gram-Charlier series of type A or Gauss-Hermite series, were initially applied to approximate probability density functions \citep{Sil:86} and later to problems in Astrophysics \citep{Bli:98}.
The series raise many questions of physical interest, for example, in the study of the harmonic oscillator in Quantum Mechanics and of equatorial waves in Dynamic Meteorology and Oceanography~\citep{Boy:84}.

\paragraph{Norm convergence of Hermite series}

In 1965, Askey and Wainger \citep{Ask:65} showed that for $4/3<p<4$ one has
\begin{equation} \label{E:p-convergence}
	\lim_{n\to\infty} \|S_nf-f\|_p = 0
		\quad\text{whenever}\quad
	\|f\|_p<\infty,
\end{equation}
whereas the same does not hold for $1\le p \le 4/3$ or $p\ge 4$.

A special case of our general result asserts that for all $1<p<\infty$
one has
\begin{equation} \label{E:p-convergence-local}
	\lim_{n\to\infty} \|\chi_{n} S_n(f\chi_{n})-f\|_p = 0
		\quad\text{if}\quad
	\|f\|_p<\infty,
\end{equation}
where the $\chi_n$ are characteristic functions of sufficiently slowly increasing subintervals of $\R$, namely
\begin{equation} \label{E:chn-def}
	\chi_n = \chi_{(-T_n,T_n)}
		\quad\text{and}\quad
	T_n = o(n^{1/34})\text{ as $n\to\infty$}.
\end{equation}

We work with an expression for the truncated partial sum
\begin{equation} \label{E:truncated-Sn}
	\chn S_n(f\chn),
		\quad n=0,1,\dots,
\end{equation}
due to G.~Sansone, which is a refinement of the one employed
by J.V.~Uspensky \citep{Usp:26} to prove his classical pointwise
convergence theorems for the Hermite series.
The heart of the truncated partial sum operator~\eqref{E:truncated-Sn}
is the truncated Dirichlet operator
\begin{equation*}
	F_N (f\chn)(x)
		= \int_{-T_n}^{T_n} \frac{\sin\bigl(N(x-y)\bigr)}{x-y} f(y)\,\d y
		\quad\text{for $x \in\R$}
\end{equation*}
with $N\in\R$ being tied to $n$.
It is through this fact that we are
able to connect the truncated partial sum \eqref{E:truncated-Sn} to 
another well-known operator,
namely to the Stieltjes transform $S$
defined for the measurable function $g$ on $(0,\infty)$ by
\begin{equation*}
	S g(t) = \int_0^\infty \frac{g(s)}{t+s}\,\d s
		\quad\text{for $t> 0$}.
\end{equation*}

\paragraph{General result}

The general result involves so-called rearrangement-invariant (\ri) norms,
a general framework at the intersection of measure theory and functional analysis
which abstracts Lebesgue, Lorentz, Orlicz and many more customary function spaces;
see \eg~\citep{Ben:88}.
The definitions are summarized at the beginning of Section~\ref{S:ineq}.

\begin{theorem}\label{T:ri}
Let $\rho_X$ and $\rho_Y$ be \ri norms
such that $\varphi_X(0+)=0$ and $X_b(\R) \subseteq Y(\R)$.
Suppose that $\chn$ is taken as in \eqref{E:chn-def}.
Then the following assertions are equivalent:
\begin{enumerate}
\item\label{en:convergence-ri}
The norm $\rho_Y$ obeys \eqref{E:logY} and
\begin{equation*}
	\lim_{n\to\infty} \bigl\| \chn S_n(f\chn) - f \bigr\|_{Y(\R)} = 0
		\quad\text{for every $f\in X_b(\R)$};
\end{equation*}
\item\label{en:S-ri}
One has
\begin{equation*}
	\bigl\| S g \bigr\|_{Y\dom} \lesssim \|g\|_{X\dom}
		\quad\text{for every $g\in X\dom$},
\end{equation*}
in which $S$ is the Stieltjes transform.
\end{enumerate}
\end{theorem}

Recall that $\varphi_X$ denotes the fundamental function, \ie $\varphi_X(t)$ is the norm of $\chi_E$
for any set of measure $t$ and $X_b(\R)$ is the closure of the simple functions in $X(\R)$.
The assumption can be equivalently formulated as the closure in $X$ of smooth and compactly supported functions contained in $Y$; see Lemma~\ref{L:Dclosure}.

The convergence of the truncated partial sum~\eqref{E:truncated-Sn} therefore reduces to the study of the
boundedness of Stieltjes transform $S$.
In the context of \ri spaces, an exhaustive treatment of optimal spaces for $S$ is available in~\citep[Chapter~5]{Edm:20}.
Examples of optimal pairs of Lorentz-Zygmund spaces are contained in \citep[Theorem~5.3]{Edm:20}.

\paragraph{Orlicz modular convergence}

One can apply Theorem~\ref{T:ri} to Orlicz norms by employing a well-known analysis of the Stieltjes transform between Orlicz spaces; see \eg~\citep{Cia:99a}.
Our next result characterizes the convergence when the Orlicz norm in \ref{en:convergence-ri} of Theorem~\ref{T:ri} is replaced by the corresponding Orlicz modular.

Recall that a function $A\colon[0,\infty)\to[0,\infty]$ is called a Young function if
$A$ is left-continuous, convex, satisfying $A(0)=0$, and not constant in $(0,\infty)$.
Next, $\tilde A(t) = \sup\{\tau t - A(\tau): \tau\ge 0\}$ denotes the Young function conjugate to $A$.
For a Young function $A$, denote
\begin{equation} \label{E:EA-def}
	E^A(\R) = 
		\left\{ 
			f\colon\R\to\R\,\text{measurable}:
				\textstyle\int_{\R} A\bigl(c|f|\bigr) < \infty
				\text{ for all $c>0$}
		\right\}.
\end{equation}
If $A$ is finite-valued then $E^A(\R)$ coincides with the the closure of smooth,
compactly supported functions on $\R$ in the norm of the corresponding Orlicz
space; see Section~\ref{S:Orlicz} for details.

\begin{theorem} \label{T:main-Orlicz}
Let $A$ and $B$ be Young functions with
$B(t)\le A(ct)$ for some $c>0$ independent of $t>0$.
Assume that $A$ is finite-valued and $B$ obeys
\begin{equation} \label{E:logB}
	\int_{0}^\infty B\Bigl(\kappa\log\left(1+\tfrac{1}{t}\right)\Bigr)\,\d t < \infty
		\quad\text{for some $\kappa>0$}.
\end{equation}
Then, given $\chn$ as in \eqref{E:chn-def},
the following are equivalent:
\begin{enumerate}
\item\label{en:truncated-convergence}
\begin{equation*}
	\lim_{n\to\infty} \int_{\R} B\bigl(\lambda|\chn S_n(f\chn)-f|\bigr) = 0
		\quad\text{for every $f\in E^A(\R)$ and $\lambda>0$};
\end{equation*}

\item\label{en:AB-conditions}
There exists $K>0$ such that
\begin{equation*}
	\int_{0}^t \frac{B(s)}{s^2}\,\d s \le \frac{A(Kt)}{t}
		\qquad\text{and}\qquad
	\int_{0}^t \frac{\tilde A(s)}{s^2}\,\d s \le \frac{\tilde B(Kt)}{t}
		\quad\text{for $t>0$}.
\end{equation*}
\end{enumerate}
\end{theorem}

\begin{remark}
If \ref{en:AB-conditions} holds, then the theorem's assumptions are satisfied.
Indeed, by monotonicity of~$B$,
\begin{equation*}
	\int_{0}^t \frac{B(s)}{s^2}\,\d s
		\ge \int_{t/2}^t \frac{B(s)}{s^2}\,\d s
		\ge B\left( \tfrac{t}{2} \right) \int_{t/2}^t \frac{\d s}{s^2}
		= \tfrac{1}{t} B\left( \tfrac{t}{2} \right)
  \quad\text{for $t>0$},
\end{equation*}
and the first inequality in \ref{en:AB-conditions}, in fact, implies
that $B(t)\le A(2Kt)$ for $t>0$.
Furthermore, \ref{en:AB-conditions} also yields \eqref{E:logB}
as shown in the proof of Theorem~\ref{T:main-Orlicz}.
\end{remark}

Let us turn our attention to some special cases and examples.
A Young function $A$ is said to satisfy the $\Delta_2$
condition if there is $C>0$ such that $A(2t)\le CA(t)$
for $t>0$. We write $A\in\Delta_2$. It is well known that
\begin{equation*}
	A\in\Delta_2
	\quad\text{if and only if}\quad
	\int_{0}^{t} \frac{\tilde A(s)}{s^2}
		\le \frac{\tilde A(Kt)}{t}
		\quad\text{for $t>0$}
\end{equation*}
for some $K>0$.
Next, $A$ obeys the $\nabla_2$ condition if there is $c>0$ such that $A(t)\le\frac{1}{2c}A(ct)$ for $t>0$.
We denote this by $A\in\nabla_2$.
It is immediate that $A\in\nabla_2$ if and only if $\tilde A\in\Delta_2$.
Another important role of the $\Delta_2$ condition arises in connection with the $E^A(\R)$ space.
Namely,
\begin{equation*}
	E^A(\R) = L^A(\R)
	\quad\text{if and only if}\quad
	A\in\Delta_2,
\end{equation*}
where $L^A(\R)$ is an Orlicz space defined by
\begin{equation*}
	L^A(\R) =
	\left\{
		f\colon\R\to\R\,\text{measurable}:
			\textstyle\int_{\R} A\bigl(c|f|\bigr) <\infty \text{ for some $c>0$}
	\right\}.
\end{equation*}
This yields an immediate consequence.

\begin{corollary} \label{C:Delta-2}
Let $A$ be a Young function obeying $A\in\Delta_2$ and $A\in\nabla_2$.
Then,
\begin{equation*}
	\lim_{n\to\infty} \int_{\R} A\bigl(\lambda|\chn S_n(f\chn)-f|\bigr) = 0
		\quad\text{for every $f\in L^A(\R)$ and $\lambda>0$}.
\end{equation*}
\end{corollary}

\begin{example}
A typical example of a Young function satisfying the hypothesis of
Corollary~\ref{C:Delta-2} is the function $A(t)=t^p$ for $t>0$ with
$p\in(1,\infty)$. Here, $L^A(\R)=L^p(\R)$, a customary Lebesgue space, and
Corollary~\ref{C:Delta-2} reads as the special case stated in~\eqref{E:p-convergence-local}.

Perturbations $L^A(\R)$ of the $L^p(\R)$ spaces, in which
\begin{equation} \label{E:Lp-log}
	A(t)=
	\begin{cases}
		t^p (\log\frac1t)^{\alpha_0}
			& \text{near zero}
			\\
		t^p(\log t)^{\alpha_\infty}
			& \text{near infinity}
	\end{cases}
\end{equation}
with $p\in(1,\infty)$ and $\alpha_0,\alpha_\infty\in\R$ also obey
$A\in\Delta_2$ and $A\in\nabla_2$ and can serve as examples of Young functions
to which Corollary~\ref{C:Delta-2} applies as well.
\end{example}

Limiting cases of \eqref{E:Lp-log} as $p\to 1^+$ and $p\to\infty$ are captured in the next example.

\begin{example}
Let $A$ and $B$ be Young functions satisfying
\begin{equation*}
	A(t)=
	\begin{cases}
		t(\log\tfrac1t)^{\alpha_0+1}
			& \text{near zero}
			\\
		t(\log t)^{\alpha_\infty+1}
			& \text{near infinity}
	\end{cases}
	\quad\text{and}\quad
	B(t)=
	\begin{cases}
		t(\log\tfrac1t)^{\alpha_0}
			& \text{near zero}
			\\
		t(\log t)^{\alpha_\infty}
			& \text{near infinity},
	\end{cases}
\end{equation*}
in which $\alpha_0<-1$ and $\alpha_\infty\ge0$.
Then $A$ with $B$ satisfy the conditions~\ref{en:AB-conditions}
of Theorem~\ref{T:main-Orlicz}.

Another admissible example of Young functions $A$ and $B$ is
\begin{equation*}
	A(t)=
	\begin{cases}
		\exp(-t^{\beta_0})
			& \text{near zero}
			\\
		\exp t^{\beta_\infty}
			& \text{near infinity}
	\end{cases}
	\quad\text{and}\quad
	B(t)=
	\begin{cases}
		\exp\left(-t^{\frac{\beta_0}{1-\beta_0}}\right)
			& \text{near zero}
			\\
		\exp\left(t^{\frac{\beta_\infty}{1+\beta_\infty}}\right)
			& \text{near infinity}
	\end{cases}
\end{equation*}
with $\beta_0\in(0,1)$ and $\beta_\infty>0$.
\end{example}

The coefficients $c_k$ of the partial sums $S_n(f\chn)$ depend, of course, on $n$.
We state now two results which replace \ref{en:convergence-ri} in Theorem~\ref{T:ri} by assertions involving single Hermite series.

\begin{corollary} \label{C:single-1}
Let $A$ and $B$ be Young functions satisfying the hypothesis of
Theorem~\ref{T:main-Orlicz} and assumption \ref{en:AB-conditions}.
Suppose that $f\in E^A(\R)$.
Then given $\varepsilon>0$, there exist a smooth compactly supported 
function $h$ and $n_0\in\N$ such that
\begin{equation*}
	\int_{\R} B\bigl(|\chi_n S_n h - f|\bigr) < \varepsilon
\end{equation*}
for $n\ge n_0$, where $\chi_n$ are as in \eqref{E:chn-def}.
\end{corollary}

Next corollary asserts that any suitable $f$ can be obtained from a
convergent Hermite series by multiplying the sum of that series by a fixed
function. As our fixed function, we choose
\begin{equation} \label{E:g-def}
	g(x) = \frac{1}{1+x^{36}}
		\quad\text{for $x\in\R$}.
\end{equation}

\begin{corollary} \label{C:single-2}
Let $A$ and $B$ be Young functions satisfying condition~\ref{en:AB-conditions} of
Theorem~\ref{T:main-Orlicz}.
Let $g$ be the function from \eqref{E:g-def} and let $\chi_n$ be as in \eqref{E:chn-def}.
Then
\begin{equation*}
	\lim_{n\to\infty}
		\int_{\R} B\bigl(|\chi_n S_n (gf) - gf|\bigr) = 0
		\quad\text{for every $f\in E^A(\R)$}.
\end{equation*}
\end{corollary}

The general result, Theorem~\ref{T:ri}, is proved in Sec.~\ref{S:general} and applied to
the Orlicz context in Sec.~\ref{S:Orlicz}.  Before that, though, we obtain estimates
of certain terms in an expression for truncated partial sum
\eqref{E:truncated-Sn} due to Sansone in Sec.~\ref{S:sansone}.
These estimates
are then used in certain pointwise and norm inequalities in Sec.~\ref{S:ineq}
to get technical results in~Sec.~\ref{S:density} used in the proofs of all the
statements presented.

\section{The Sansone estimates}
\label{S:sansone}

\noindent
By $A \lesssim B$ and $A \gtrsim B$ we mean that $A \le C\,B$ and $A \ge C\,B$,
respectively, where $C$ is a positive constant independent of the appropriate
quantities involved in $A$ and $B$.
For brevity, we write $\chi_T = \chi_{(-T,T)}$, $T>0$.
We will make use of another well-known classical operator, namely the Hilbert transform
given for a suitable measurable $f\colon\R\to\R$ by
\begin{equation*}
	Hf(x) = \text{p.v.}\int_{\R} \frac{f(y)}{x-y}\,\d y
		\quad\text{for $x \in\R$}
\end{equation*}
whenever the integral exists \ae.

Now, we present the pointwise estimates of the partial sums of the Hermite series, given in the monograph by G.~Sansone~\citep{San:91}.
We keep most of Sansone's notation, but we also make several adjustments.
According to \citep[p.~372, Eq.~(5)]{San:91}, we have
\begin{equation}
	S_{n} (\fT)(x)
		= \sqrt{\frac{n+1}{2}} \int_{-T}^{T} k_n(x,y)f(y)\,\d y,
		\label{eq:SEsum_n}
\end{equation}
where
\begin{equation}
	k_n(x,y) = \frac{h_n(x)h_{n+1}(y) - h_n(y)h_{n+1}(x)}{x-y}.
		\label{eq:kern_n}
\end{equation}
By \citep[p.~325, Eq.~(14$_1$) and (14$_2$)]{San:91},
we can express the Hermite functions for odd and even $n$ by different formulas as
\begin{equation*}
	h_{n}(x) =
	\begin{cases}
		\displaystyle
		h_{n}(0)
			\biggl[  \cos \bigl( \sqrt{2n+1}\, x \bigr)
						+ \frac{x^3}{6}\frac{\sin \bigl( \sqrt{2n+1}\, x \bigr)}{\sqrt{2n+1}}
						+ R(n,x)
			\biggr]
			& \text{if $n$ is even}
			\\[\bigskipamount]
		\displaystyle
		\frac{h'_{n}(0)}{\sqrt{2n+1}}
			\biggl[  \sin \bigl( \sqrt{2n+1}\, x \bigr)
						- \frac{x^3}{6} \frac{ \cos \bigl( \sqrt{2n+1}\, x \bigr)}{\sqrt{2n+1}}
						+ R(n,x)
			\biggr]
			& \text{if $n$ is odd},
	\end{cases}
\end{equation*}
where the remainder $R$ satisfies the estimate
\begin{equation*}
	|R(n,x)| \lesssim \omega(n,x),
\end{equation*}
in which
\begin{equation} \label{E:omega-def}
	\omega(x,n) = x^2 (x^4+1) n^{-1} + x^{17/2}n^{-5/4},
\end{equation}
as follows by \citep[p.~327, Eq.~(15$_1$), (15$_2$) and p.~374]{San:91}.

Now, we compute the kernel $k_{n}$ if $n$ is even.
We obtain (compare with \citep[p.~373, Eq.~(7)]{San:91})
\begin{align} \label{eq:bigproducts}
	\begin{split}
	\sqrt{\frac{n+1}{2}} (x-y) k_{n}(x,y) =
		\\
		- c_{n}
		& \biggl[  \cos \bigl( \sqrt{2n+1}\, x \bigr)
					+ \frac{x^3}{6}\frac{\sin \bigl( \sqrt{2n+1}\, x \bigr)}{\sqrt{2n+1}}
					+ R(n,x)
		\biggr]
			\\
			\times
		& \biggl[  \sin \bigl( \sqrt{2n+3}\, y \bigr)
					- \frac{y^3}{6} \frac{ \cos \bigl( \sqrt{2n+3}\, y \bigr)}{\sqrt{2n+3}}
					+ R(n+1,y)
		\biggr]
			\\
		+ c_{n}
		& \biggl[  \cos \bigl( \sqrt{2n+1}\, y \bigr)
					+ \frac{y^3}{6}\frac{\sin \bigl( \sqrt{2n+1}\, y \bigr)}{\sqrt{2n+1}}
					+ R(n,y)
		\biggr]
			\\
			\times
		& \biggl[  \sin \bigl( \sqrt{2n+3}\, x \bigr)
					- \frac{x^3}{6} \frac{ \cos \bigl( \sqrt{2n+3}\, x \bigr)}{\sqrt{2n+3}}
					+ R(n+1,x)
		\biggr].
	\end{split}
\end{align}
Here, by
\citep[p.~373, Eq.~(8)]{San:91},
\begin{equation} \label{E:cn-def}
	c_{n} = -\frac1\pi\sqrt{\frac{n+1}{2}}\frac{h_{n}(0)\,h'_{n+1}(0)}{\sqrt{2n+3}}
		= 1 + \frac{\varepsilon}{6n}
		\quad\text{and}\quad |\varepsilon| < 3.
\end{equation}
Expanding the brackets in \eqref{eq:bigproducts} and plugging them into
\eqref{eq:SEsum_n}, we get
\begin{equation} \label{E:SEsum}
	S_{n} (\fT)(x)
		 = c_{n}\sum_{k=1}^{7} \int_{-T}^{T} \frac{ K_k^{(n)}(x,y)}{ x-y }f(y)\,\d y
\end{equation}
in which
\begin{align} \label{E:K1}
	\begin{split}
	K_1^{(n)}(x,y)
		& = \sin \bigl( \sqrt{2n+3}\, x \bigr)\cos \bigl( \sqrt{2n+1}\, y \bigr)
			\\
		&\quad - \sin \bigl( \sqrt{2n+3}\, y \bigr)\cos \bigl( \sqrt{2n+1}\, x \bigr),
	\end{split}
\end{align}
\begin{align*}
	K_2^{(n)}(x,y)
		&= -\frac{1}{6\sqrt{2n+1}}
			\Bigl[ - x^3 \sin \bigl( \sqrt{2n+1}\, x \bigr)\sin \bigl( \sqrt{2n+3}\, y \bigr)
			\\
		& \quad	 + y^3 \sin \bigl( \sqrt{2n+1}\, y \bigr)\sin \bigl( \sqrt{2n+3}\, x \bigr)
			\Bigr],
\end{align*}
\begin{align*}
	K_3^{(n)}(x,y)
		&= \frac{1}{6\sqrt{2n+3}}
			\Bigl[ y^3 \cos \bigl( \sqrt{2n+1}\, x \bigr)\cos \bigl( \sqrt{2n+3}\, y \bigr)
			\\
		&\quad - x^3\cos \bigl( \sqrt{2n+1}\, y \bigr)\cos \bigl( \sqrt{2n+3}\, x \bigr)
			\Bigr],
\end{align*}
\begin{align*}
	K_4^{(n)}(x,y)
		&= \frac{x^3y^3}{36\sqrt{(2n+1)(2n+3)}}
			\Bigl[ \sin \bigl( \sqrt{2n+1}\, y \bigr)\cos \bigl( \sqrt{2n+3}\, x \bigr)
			\\
		&\quad	-\cos \bigl( \sqrt{2n+1}\, y \bigr)\sin \bigl( \sqrt{2n+3}\, x \bigr)
			\Bigr],
\end{align*}
\begin{align*}
	K_5^{(n)}
		& = R(n+1,x)\cos \bigl( \sqrt{2n+1}\,y \bigr)
			- R(n+1,y)\cos \bigl( \sqrt{2n+1}\,x \bigr)
			\\
		& \quad
			+ R(n,x)\sin \bigl( \sqrt{2n+3}\,y \bigr)
			- R(n,y)\sin \bigl( \sqrt{2n+3}\,x \bigr)
		,
\end{align*}
\begin{align*}
	K_6^{(n)}
		& = R(n+1,y)\frac{x^3}{6} \frac{\sin \bigl( \sqrt{2n+1}\,x \bigr)}{\sqrt{2n+1}} 
			- R(n+1,x)\frac{y^3}{6} \frac{\sin \bigl( \sqrt{2n+1}\,y \bigr)}{\sqrt{2n+1}} 
			\\
		& \quad
			+ R(n,y)\frac{x^3}{6} \frac{\cos \bigl( \sqrt{2n+3}\,x \bigr)}{\sqrt{2n+3}} 
			- R(n,x)\frac{y^3}{6} \frac{\cos \bigl( \sqrt{2n+3}\,y \bigr)}{\sqrt{2n+3}} 
\end{align*}
and
\begin{equation*}
	K_7^{(n)}
		= R(n+1,x) R(n,y) - R(n+1,y) R(n,x).
\end{equation*}
If $n$ is odd, the kernel $k_{n}$ and also the partial sum $S_{n}(\fT)$
can be represented in an analogous way. We omit the details.

Let us estimate all the terms from \eqref{E:SEsum}.
We have
\begin{align} \label{E:SEK1}
	\begin{split}
	\biggl| \int_{-T}^{T} \frac{K_1^{(n)}(x,y)}{x-y} f(y)\,\d y \biggr|
		& \le 
			\biggl| \int_{-T}^{T} \frac{ \cos \bigl( \sqrt{2n+1}\, y \bigr) }{x-y} f(y)\,\d y \biggr|
		+ \biggl| \int_{-T}^{T} \frac{ \sin \bigl( \sqrt{2n+3}\, y \bigr) }{x-y} f(y)\,\d y \biggr|
				\\
		& = \bigl| H (m_1 \fT)(x) \bigr|
				+ \bigl| H (m_2 \fT)(x) \bigr|,
				\quad x\in(-T,T),
	\end{split}
\end{align}
where $\|m_j\|_{L^\infty(-T,T)} \le 1$, $j=1,2$.
Next
\begin{align} \label{E:SEK2}
	\begin{split}
	\biggl| \int_{-T}^{T} \frac{K_2^{(n)}(x,y)}{x-y} f(y)\,\d y \biggr|
		& \le \frac{T^3}{6\sqrt{2n+1}}
				\biggl| \int_{-T}^{T} \frac{ \sin \bigl( \sqrt{2n+3}\, y \bigr) }{x-y} f(y)\,\d y \biggr|
				\\
		&\quad 	+ \frac{1}{6\sqrt{2n+1}}
				\biggl| \int_{-T}^{T} \frac{ y^3 \sin \bigl( \sqrt{2n+1}\, y \bigr) }{x-y} f(y)\,\d y \biggr|
			 \\
		& \lesssim \frac{T^3}{\sqrt{n}}
				\bigl| H (m_2 \fT)(x) \bigr|
			+ \frac{1}{\sqrt{n}}
				\bigl| H (m_3 \fT)(x) \bigr|,
				\quad x\in(-T,T),
	\end{split}
\end{align}
where $\|m_3\|_{L^\infty(-T,T)} \le T^3$,
and similarly
\begin{align} \label{E:SEK3}
	\begin{split}
	\biggl| \int_{-T}^{T} \frac{K_3^{(n)}(x,y)}{x-y} f(y)\,\d y \biggr|
		& \le \frac{1}{6\sqrt{2n+3}}
				\biggl| \int_{-T}^{T} \frac{ y^3 \cos \bigl( \sqrt{2n+3}\, y \bigr) }{x-y} f(y)\,\d y \biggr|
				\\
		&\quad 	+ \frac{T^3}{6\sqrt{2n+3}}
				\biggl| \int_{-T}^{T} \frac{ \cos \bigl( \sqrt{2n+1}\, y \bigr) }{x-y} f(y)\,\d y \biggr|
			 \\
		& \lesssim \frac{1}{\sqrt{n}}
				\bigl| H (m_4 \fT)(x) \bigr|
			+ \frac{T^3}{\sqrt{n}}
				\bigl| H (m_1 \fT)(x) \bigr|,
				\quad x\in(-T,T),
	\end{split}
\end{align}
where $\|m_4\|_{L^\infty(-T,T)} \le T^3$.
For $K_4$, we have
\begin{align} \label{E:SEK4}
	\begin{split}
	\biggl| \int_{-T}^{T} \frac{K_4^{(n)}(x,y)}{x-y} f(y)\,\d y \biggr|
		& \le \frac{T^3}{36\sqrt{(2n+1)(2n+3)}}
				\biggl| \int_{-T}^{T} \frac{ y^3 \sin \bigl( \sqrt{2n+1}\, y \bigr) }{x-y} f(y)\,\d y \biggr|
				\\
		&\quad 	+ \frac{T^3}{36\sqrt{(2n+1)(2n+3)}}
				\biggl| \int_{-T}^{T} \frac{ y^3 \cos \bigl( \sqrt{2n+1}\, y \bigr) }{x-y} f(y)\,\d y \biggr|
			 \\
		& \lesssim \frac{T^3}{n}
			\Bigl[
				\bigl| H (m_3 \fT)(x) \bigr|
				+
				\bigl| H (m_4 \fT)(x) \bigr|
			\Bigr],
				\quad x\in(-T,T),
	\end{split}
\end{align}
and for $K_5$, we have
\begin{align} \label{E:SEK5}
	\begin{split}
	\biggl| \int_{-T}^{T} \frac{K_5^{(n)}(x,y)}{x-y} f(y)\,\d y \biggr|
		& \le \bigl| R(n+1,x) \bigr|
				\biggl| \int_{-T}^{T} \frac{ \cos \bigl( \sqrt{2n+1}\, y \bigr) }{x-y} f(y)\,\d y \biggr|
				\\
		&\quad + \bigl| R(n,x) \bigr|
				\biggl| \int_{-T}^{T} \frac{ \sin \bigl( \sqrt{2n+3}\, y \bigr) }{x-y} f(y)\,\d y \biggr|
				\\
		&\quad 	+
				\biggl| \int_{-T}^{T} \frac{ R(n+1,y) }{x-y} f(y)\,\d y \biggr|
						+
				\biggl| \int_{-T}^{T} \frac{ R(n,y) }{x-y} f(y)\,\d y \biggr|
			 \\
		& \lesssim \omega(T,n)
			\Bigl[
				\bigl| H (m_1 \fT)(x) \bigr|
				+
				\bigl| H (m_2 \fT)(x) \bigr|
			\Bigr]
			 \\
		& \quad +
			\Bigl[
				\bigl| H (m_5 \fT)(x) \bigr|
				+
				\bigl| H (m_6 \fT)(x) \bigr|
			\Bigr],
				\quad x\in(-T,T),
	\end{split}
\end{align}
where $\|m_j\|_{L^\infty(-T,T)} \le \omega(T,n)$, $j=5,6$.
Next
\begin{align} \label{E:SEK6}
	\begin{split}
	\biggl| \int_{-T}^{T} \frac{K_6^{(n)}(x,y)}{x-y} f(y)\,\d y \biggr|
		& \le \frac{\bigl| R(n+1,x) \bigr|}{6\sqrt{2n+1}}
				\biggl| \int_{-T}^{T} \frac{ y^3 \sin \bigl( \sqrt{2n+1}\, y \bigr) }{x-y} f(y)\,\d y \biggr|
				\\
		&\quad + \frac{ \bigl| R(n,x) \bigr|}{6\sqrt{2n+3}}
				\biggl| \int_{-T}^{T} \frac{ y^3 \cos \bigl( \sqrt{2n+3}\, y \bigr) }{x-y} f(y)\,\d y \biggr|
				\\
		&\quad 	+ \frac{T^3}{6\sqrt{2n+1}}
				\biggl| \int_{-T}^{T} \frac{ R(n+1,y) }{x-y} f(y)\,\d y \biggr|
			 \\
		&\quad 	+ \frac{T^3}{6\sqrt{2n+3}}
				\biggl| \int_{-T}^{T} \frac{ R(n,y) }{x-y} f(y)\,\d y \biggr|
			 \\
		& \lesssim \frac{\omega(T,n)}{\sqrt{n}}
			\Bigl[
				\bigl| H (m_3 \fT)(x) \bigr|
				+
				\bigl| H (m_4 \fT)(x) \bigr|
			\Bigr]
			 \\
		& \quad + \frac{T^3}{\sqrt{n}}
			\Bigl[
				\bigl| H (m_5 \fT)(x) \bigr|
				+
				\bigl| H (m_6 \fT)(x) \bigr|
			\Bigr],
				\quad x\in(-T,T).
	\end{split}
\end{align}
And finally
\begin{align} \label{E:SEK7}
	\begin{split}
	\biggl| \int_{-T}^{T} \frac{K_7^{(n)}(x,y)}{x-y} f(y)\,\d y \biggr|
		& \le \bigl| R(n+1,x) \bigr|
				\biggl| \int_{-T}^{T} \frac{ R(n,y) }{x-y} f(y)\,\d y \biggr|
				\\
		&\quad + \bigl| R(n,x) \bigr|
				\biggl| \int_{-T}^{T} \frac{ R(n+1,y) }{x-y} f(y)\,\d y \biggr|
				\\
		& \lesssim \omega(T,n)
			\Bigl[
				\bigl| H (m_5 \fT)(x) \bigr|
				+
				\bigl| H (m_6 \fT)(x) \bigr|
			\Bigr],
				\quad x\in(-T,T).
	\end{split}
\end{align}
Note that all of the above integrals are to be taken in the same Cauchy principal value sense as the one for the Hilbert transform.

Summarizing the estimates in \eqref{E:SEsum}, and using that $c_n$ is bounded, we have
\begin{align} \label{E:SE-summary}
	\begin{split}
        |S_{n} (\fT)(x)|
            \lesssim \sum_{k=1}^7\,
            \biggl| \int_{-T}^{T} \frac{K_7^{(n)}(x,y)}{x-y} f(y)\,\d y \biggr|
    		  \lesssim \sum_{j=1}^6 \gamma_j \bigl| H (m_j \fT)(x) \bigr|,
        \quad x\in(-T,T),
	\end{split}
\end{align}
in which
\begin{equation} \label{E:SE-m-bounds}
    \gamma_j
        = 1 + \frac{T^3}{\sqrt{n}} + \omega(T,n)
    \text{ for $j=1,2,5,6$}
    \quad\text{and}\quad
    \gamma_j
        = \frac1{\sqrt{n}} + \frac{T^3}{n} + \frac{\omega(T,n)}{\sqrt{n}}
    \text{ for $j=3,4$}.
\end{equation}
Recall that
$\|m_j\|_{L^\infty(-T,T)} \le 1$ for $j=1,2$,
$\|m_j\|_{L^\infty(-T,T)} \le T^3$ for $j=3,4$ and
$\|m_j\|_{L^\infty(-T,T)} \le \omega(T,n)$ for $j=5,6$,
and that, from~\eqref{E:omega-def},
\begin{equation} \label{E:omega-bound}
    \omega(T,n)\lesssim \frac{T^{{17}/{2}}}{n}.
\end{equation}

We conclude the section with an alternative expression for $S_{n}(\fT)$.
Using trigonometric identities, we have
\begin{align*} 
	K_1^{(n)}(x,y)
		= \sin \bigl( N(x-y) \bigr)
	       + \cos \bigl(N(x+y)\bigr) \sin \bigl( (x-y)/2N \bigr)
		  - 2 \sin^2 \bigl( (x+y)/4N \bigr)
			\sin\bigl(N(x-y)\bigr),
\end{align*}
whence with $x\in(-T,T)$, one has
\begin{align*}
	\int_{-T}^{T} \frac{K_1^{(n)}(x,y)}{x-y}f(y)\,\d y
		& = \int_{-T}^{T} \frac{\sin\bigl(N(x-y)\bigr)}{x-y}f(y)\,\d y
		  \\
		&\quad + \int_{-T}^{T}
			\frac{ \cos\bigl( N(x+y) \bigr)}{x-y} \sin\Bigl(\frac{x-y}{2N}\Bigr) f(y)\,\d y
			\\
		&\quad - \int_{-T}^{T}
			\sin^2\Bigl( \frac{x+y}{4N} \Bigr) \frac{ \sin\bigl( N(x-y) \bigr)}{ x-y }f(y)\,\d y.
\end{align*}
In sum, using \eqref{E:SEsum}, we get
\begin{align} \label{E:Sneq}
    \begin{split} 
        S_{n} (\fT)(x)
            & = \int_{-T}^{T} \frac{\sin\bigl(N(x-y)\bigr)}{x-y}f(y)\,\d y
    		  + c_{n}\sum_{k=2}^{7} \int_{-T}^{T} \frac{ K_k^{(n)}(x,y)}{ x-y }f(y)\,\d y 
              + c_n I(x) + c_n II(x)
    \end{split}
\end{align}
for $x\in(-T,T)$, where
\begin{equation} \label{E:I-def}
    I(x) = \int_{-T}^{T} \frac{ \cos\bigl( N(x+y) \bigr)}{x-y} \sin\Bigl(\frac{x-y}{2N}\Bigr) f(y)\,\d y
\end{equation}
and
\begin{equation} \label{E:II-def}
    II(x) = \int_{-T}^{T} \sin^2\Bigl( \frac{x+y}{4N} \Bigr) \frac{ \sin\bigl( N(x-y) \bigr)}{ x-y }f(y)\,\d y.
\end{equation}
Collecting the estimates \eqref{E:SEK2}--\eqref{E:SEK7} and using that $c_n$ is bounded, we get the estimate for the sum in \eqref{E:Sneq}, namely,
\begin{align} \label{E:SE-summary-2}
	\begin{split}
        \sum_{k=2}^7\,
            \biggl| \int_{-T}^{T} \frac{K_7^{(n)}(x,y)}{x-y} f(y)\,\d y \biggr|
    		  \lesssim \sum_{j=1}^6 \widetilde\gamma_j \bigl| H (m_j \fT)(x) \bigr|,
        \quad x\in(-T,T),
	\end{split}
\end{align}
where
\begin{equation} \label{E:SE-m-bounds-2}
    \widetilde\gamma_j
        = \gamma_j - 1
        = \frac{T^3}{\sqrt{n}} + \omega(T,n)
    \text{ for $j=1,2$}
    \quad\text{and}\quad
    \widetilde\gamma_j = \gamma_j
    \text{ otherwise.}
\end{equation}
Recall that all the constants involved in $\lesssim$ are absolute, independent of $n$, $T$ and $f$.

\section{Rearrangement invariant spaces and ensuing inequalities} \label{S:ineq}

\noindent
Let $\RR$ denote either $\R$ or $\R_+$. Let $\MM(\RR)$ be the class
of real-valued measurable functions on $\RR$ and $\MM^+(\RR)$ the set of
nonnegative functions in $\MM(\RR)$. For $f\in\MM(\RR)$,
$f^*\colon(0,\infty)\to[0,\infty]$ denotes the non-increasing rearrangement of
$f$ defined by
\begin{equation*}
	f^*(t) := \inf
		\bigl\{ \lambda>0:
			|\{ x\in\RR: |f(x)| > \lambda \}| \le t
		\bigr\}
	\quad\text{for $t>0$},
\end{equation*}
where $|E|$ is the Lebesgue measure of a measurable set $E\subset\RR$.

A mapping $\rho\colon\MM^+(\RR)\to[0,\infty]$ is called a re\-arr\-ange\-ment-invariant (\ri for short)\ Banach function norm on~$\MM^+(\RR)$, if for all $f$, $g$, $f_n$, $n\in\N$, in $\MM^+(\RR)$, for all constants $a\ge 0$ and for every measurable subset $E$ of $\RR$, the following properties hold:

\begin{enumerate}[label={\rm(P\arabic*)},left=0pt .. 3\parindent]
\item $\rho(f) = 0$ iff $f=0$ \ae;
			$\rho(af) = a\rho(f)$;
			$\rho(f+g) \le \rho(f)+\rho(g)$
\item	$0\le f\le g$ \ae implies $\rho(f)\le\rho(g)$
\item	$0\le f_n\uparrow f$ \ae implies $\rho(f_n)\uparrow\rho(f)$
\item $|E|<\infty$ implies $\rho(\chi_E)<\infty$
\item $|E|<\infty$ implies $\int_E f\,\d x\le c_E \rho(f)$
\item $\rho(f)=\rho(g)$ whenever $f^*=g^*$
\end{enumerate}
for some constant $0<c_E<\infty$ depending on $E$ but independent of $f$.
We also call $\rho$ just \ri norm for brevity.

Given \ri norm $\rho_X$,
the collection $X(\RR)$ of all functions $f$ in $\MM(\RR)$
for which $\|f\|_{X(\RR)}=\rho_X(|f|)$ is finite is called a rearrangement-invariant
Banach function space (or just \ri space).
For $E\subset\RR$ measurable, we write
\begin{equation*}
	\|f\|_{X(E)} = \|f\chi_E\|_{X(\RR)}.
\end{equation*}
A fundamental result of Luxemburg \citep[Chapter~2, Theorem~4.10]{Ben:88}
asserts that to every \ri space $X(\RR)$ there corresponds
an \ri space $X\dom$ such that
\begin{equation*}
	\|f\|_{X(\RR)} = \|f^*\|_{X\dom}
	\quad\text{for $f\in X(\RR)$}.
\end{equation*}
This space is called the representation
space of $X(\RR)$. Clearly, if $\RR=\R_+$, the space $X(\RR)$
and its representation space coincide.

The following estimate of Hilbert transform is well known.
It asserts that
there is a constant $C>0$ such that
for any $f\in\MM(\R)$.
\begin{equation} \label{E:HleS}
	(H f)^*(t)
		\le C\, S f^*(t)
		\quad\text{for $t\in\R_+$.}
\end{equation}
The proof can be found in
\citep[Chapter~3, Theorem~4.7]{Ben:88}. 
This inequality yields an immediate consequence.

\begin{lemma} \label{L:HleS}
Let $\rho_Z$ be an \ri norm and let $T\in(0,\infty]$.
Assume that $f,m\in\MM(\R)$. Then
\begin{equation*}
	\|H(m\fT)\|_{Z(-T,T)}
		\lesssim \|m\|_{L^\infty(-T,T)} \|Sf^*\|_{Z\dom}.
\end{equation*}
\end{lemma}

The next result shows an inequality between
Hermite partial sums and the Stieltjes transform.
Its proof is a consequence of the Sansone estimates
introduced in Sec.~\ref{S:sansone}.

\begin{lemma} \label{L:SnleS}
Let $\rho_Z$ be an \ri norm. Then for any $f\in\MM(\R)$, $T>0$ and $n\in\N$, one has
\begin{equation} \label{E:SnleS}
	\|S_n (\fT)\|_{Z(-T,T)}
		\lesssim \biggl(1+\frac{T^{17}}{\sqrt{n}}\biggr) \|S f^*\|_{Z\dom}.
\end{equation}
\end{lemma}

\begin{proof}
Using Sansone estimates \eqref{E:SE-summary} together with Lemma~\ref{L:HleS} yields
\begin{equation}
    \begin{split}
    	\|S_n(\fT)\|_{Z(-T,T)}
    		\lesssim \sum_{j=1}^{6} \gamma_j \|H(m_j\fT)\|_{Z(-T,T)}
    		\lesssim
                \Biggl(
                    \sum_{j=1}^{6} \gamma_j \|m_j\|_{L^\infty(-T,T)} 
                \Biggr)
                \|Sf^*\|_{Z\dom}.
    \end{split}
\end{equation}
Collecting estimates~\eqref{E:SE-m-bounds} and \eqref{E:omega-bound} we continue by
\begin{align*}
    \sum_{j=1}^{6} \gamma_j \|m_j\|_{L^\infty(-T,T)}
        & \lesssim
            \Biggl(
                1 + \frac{T^3}{\sqrt{n}} + \omega(T,n)
            \Biggr)
            \bigl(
                \|m_1\|_{L^\infty(-T,T)} + \|m_2\|_{L^\infty(-T,T)}
            \bigr)
            \\
        & \quad +
            \Biggl(
                \frac1{\sqrt{n}} + \frac{T^3}{n} + \frac{\omega(T,n)}{\sqrt{n}}
            \Biggr)
            \bigl(
                \|m_3\|_{L^\infty(-T,T)} + \|m_4\|_{L^\infty(-T,T)}
            \bigr)
            \\
        & \quad +
            \Biggl(
                1 + \frac{T^3}{\sqrt{n}} + \omega(T,n)
            \Biggr)
            \bigl(
                \|m_5\|_{L^\infty(-T,T)} + \|m_6\|_{L^\infty(-T,T)}
            \bigr)
            \\
        & \lesssim  
            \Biggl(
                1 + \frac{T^3}{\sqrt{n}} + \frac{T^{17/2}}{n}
            \Biggr)
            +
            \Biggl(
                \frac1{\sqrt{n}} + \frac{T^3}{n} + \frac{T^{17/2}}{n\sqrt{n}}
            \Biggr)
            T^3
            +
            \Biggl(
                1 + \frac{T^3}{\sqrt{n}} + \frac{T^{17/2}}{n}
            \Biggr)
            \frac{T^{17/2}}{n}
            \\
        & \lesssim 1 + \frac{T^{17}}{\sqrt n}.
    \qedhere
\end{align*}
\end{proof}

\begin{lemma}
\label{L:SleSn}
Let $g\in\MM^+\dom$ and
let $\chn$ be given as in \eqref{E:chn-def}.
Then
\begin{equation} \label{E:dec13}
	Sg(x)
		\lesssim \liminf_{n\to\infty} \sum_{k=1}^4 \sum_{m=k-3}^k
		\chn(x) |S_n (f_{k,m}\chi_n)(x)|
		\quad\text{for $x\in\R_+$}
\end{equation}
in which
\begin{equation} \label{E:fkm-def}
	f_{k,m}(x) = g_{k,m}(-x)\chi_{(-\infty,0)}(x)
		\quad\text{for $x\in\R$}
\end{equation}
and
\begin{equation}
	g_{k,m}(y) = g\Bigl( y - \frac{m}{4}\frac{\pi}{N} \Bigr) \chi_{I_k} (y)
		\quad\text{for $y\in\R$},
\end{equation}
where
\begin{equation} \label{E:N-def}
	N = \frac{1}{2}\bigl( \sqrt{2n+1} + \sqrt{2n+3} \bigr)
\end{equation}
and
\begin{equation}
	I_k = \bigcup_{j=0}^\infty
		\Bigl(
            \frac{k}{4}\frac{\pi}{N} + \frac{j\pi}{N},
		    \frac{k+1}{4}\frac{\pi}{N} + \frac{j\pi}{N}
        \Bigr),
		\quad k=0,1,\ldots
\end{equation}
\end{lemma}

\begin{proof}
Let $f(y)=g(-x)\chi_{(-\infty, 0)}(y)$ for $g\in\MM^+\dom$.
We aim to estimate $S_n(\fT)$ by estimating all the terms on the right-hand side in equality~\eqref{E:Sneq}.
First,
\begin{equation*}
    H(m\fT)(x)
        = \int_{-T}^{T} \frac{m(y)g(-y)}{x-y}\,\d y
        = \int_{0}^{T} \frac{m(-y)g(y)}{x+y}\,\d y,
\end{equation*}
so that
\begin{equation} \label{E:HleMInt}
    \bigl|H(m\fT)(x)\bigr|
        \le \|m\|_{L^\infty(-T,T)} \int_{0}^{T} \frac{g(y)}{x+y}\,\d y
    \quad\text{for $x\in(0,T)$.}
\end{equation}
Combining the Sansone estimates~\eqref{E:SE-summary-2} with the estimate~\eqref{E:HleMInt}, we obtain
\begin{equation*}
    \sum_{k=2}^7\,
        \biggl| \int_{-T}^{T} \frac{K_7^{(n)}(x,y)}{x-y} f(y)\,\d y \biggr|
           \lesssim \sum_{j=1}^6 \widetilde\gamma_j \bigl| H (m_j \fT)(x) \bigr|
    	   \lesssim \biggl(\sum_{j=1}^6 \widetilde\gamma_j \|m_j\|_{L^\infty(-T,T)}\biggr) 
                \int_{0}^{T} \frac{g(y)}{x+y}\,\d y
\end{equation*}
for $x\in(0,T)$.
Similarly as in the proof of Lemma~\ref{L:SnleS}, we have
\begin{equation}\label{E:sums-gm}
    \sum_{j=1}^6 \widetilde\gamma_j \|m_j\|_{L^\infty(-T,T)}
        \lesssim \frac{T^{17}}{\sqrt{n}}.
\end{equation}
Next, we estimate the terms \eqref{E:I-def} and \eqref{E:II-def}, still with our choice of $f(y)=g(-y)\chi_{(-\infty, 0)}(y)$.
With $x\in(0,T)$,
\begin{align*}
    |I(x)|
        & \le \int_0^T \biggl|\frac{\cos\bigl(N(x-y)\bigr)}{x+y}\biggr| \sin\Bigl(\frac{x+y}{2N}\Bigr) g(y)\,\d y
          \lesssim \int_0^T \frac{1}{x+y} \frac{x+y}{2N}g(y)\,\d y
            \\
        & \lesssim \frac{T}{\sqrt{n}} \int_0^T \frac{g(y)}{T}\,\d y
          \lesssim \frac{T}{\sqrt{n}} \int_0^T \frac{g(y)}{T+y}\,\d y
          \lesssim \frac{T}{\sqrt{n}} \int_0^T \frac{g(y)}{x+y}\,\d y
\end{align*}
and
\begin{align*}
    |II(x)|
        & \le \int_0^T
            \sin^2\Bigl( \frac{x-y}{4N} \Bigr)
            \bigl| \sin\bigl( N(x+y) \bigr) \bigr|
            \frac{g(y)}{x+y}\,\d y
          \lesssim \int_0^T \Bigl(\frac{x-y}{4N}\Bigr)^2\frac{g(y)}{x+y}\,\d y
          \lesssim \frac{T^2}{n} \int_0^T \frac{g(y)}{x+y}\,\d y.
\end{align*}
Applying all the estimates to equation~\eqref{E:Sneq}, we obtain
\begin{align} \label{E:Sneq-2}
    \begin{split}
    \bigl|S_{n} (\fT)(x)\bigr|
        & \gtrsim c_n \int_{-T}^{T} \frac{\sin\bigl(N(x-y)\bigr)}{x-y}f(y)\,\d y
		- \biggl| 
                c_{n}\sum_{k=2}^{7} \int_{-T}^{T} \frac{ K_k^{(n)}(x,y)}{ x-y }f(y)\,\d y 
          \biggr|
        - |c_n I(x)| - |c_n II(x)|
            \\ 
        & \gtrsim c_n \int_{0}^{T} \sin\bigl(N(x+y)\bigr)\frac{g(y)}{x+y}\,\d y
            - \frac{T^{17}}{\sqrt{n}} \int_0^T \frac{g(y)}{x+y}\,\d y.
    \end{split}
\end{align}

Now, fix $k\in\{1,2,3,4\}$ and choose $\sigma(K)\in\{1,2,3,4\}$ so that $k+\sigma(k)=1 \pmod 4$.
Then, for $y\in I_k$ and $x\in I_{\sigma(k)}$ one has
\begin{equation*}
	\frac{\pi}{4N} \le x + y \le \frac{3\pi}{4N}
		\quad \Big( \mathop{\mathrm{mod}}\nolimits \frac{\pi}{ N } \Bigr)
\end{equation*}
and $\sin\bigl(N(x+y)\bigr)\ge \sqrt{2}/2$.
Therefore, taking $f=f_{k,m}$ and $T=T_n$, we have
\begin{equation*}
	\int_{0}^{T_n} \sin \bigl(N(x+y)\bigr) \frac{g_{k,m}(y)}{x+y}\,\d y
		\gtrsim \int_0^{T_n} \frac{g_{k,m}(y)}{x+y}\,\d y
    \quad\text{for $x\in I_{\sigma(k)}$}
\end{equation*}
and \eqref{E:Sneq-2} yields
\begin{equation*}
    \bigl|S_{n} (f_{k,m}\chn)(x)\bigr|
        \gtrsim \Bigl(c_n - \frac{T_n^{17}}{\sqrt{n}}\Bigr)
            \int_0^{T_n} \frac{g_{k,m}(y)}{x+y}\,\d y
        \quad\text{for $x\in I_{\sigma(k)}$}.
\end{equation*}
Further, by the change of variables, we have for $x\in I_{\sigma(k)}$ that
\begin{align*}
    \int_0^{T_n} \frac{g_{k,m}(y)}{x+y}\,\d y
	   = \int_{0}^{T_n}
			\frac{g\left(y - \frac{m}{4}\frac{\pi}{N} \right)}{x+y}\chi_{I_k}(y)\,\d y
	   = \int_{0}^{T_n-\frac{m\pi}{4N}}
			\frac{g(y)\chi_{I_{k-m}}(y)}{x+y+\frac{m}{4}\frac{\pi}{N}}\,\d y.
\end{align*}
Moreover, since $x\in I_{\sigma(k)}$ it is $x\ge \frac{\pi}{4N}$
and since $m\le k\le 4$, it follows that
\begin{equation*}
	x + y + \frac{m}{4}\frac{\pi}{N} \le x + y +mx \le 5(x+y)
\end{equation*}
for any $y\ge 0$. Thus
\begin{equation*}
	\bigl| S_n (f_{k,m}\chn)(x) \bigr|
        \gtrsim \Bigl(c_n - \frac{T_n^{17}}{\sqrt{n}}\Bigr)
            \int_{0}^{T_n-\frac{\pi}{N}}
			\frac{g(y)}{x+y}\chi_{I_{k-m}}(y)\,\d y.
\end{equation*}
Summing over $m=k-3,\dots,k$, we obtain
\begin{equation*}
    \Bigl(c_n - \frac{T_n^{17}}{\sqrt{n}}\Bigr)
	\int_0^{T_n-\frac{\pi}{N}} \frac{g(y)}{x+y}\,\d y
		\lesssim \sum_{m=k-3}^k
			\bigl| S_n (f_{k,m}\chn)(x) \bigr|
		\quad\text{for $x\in I_{\sigma(k)}$}
\end{equation*}
and another summing over $k$ yields
\begin{equation*}
    \Bigl(c_n - \frac{T_n^{17}}{\sqrt{n}}\Bigr)
	\int_0^{T_n-\frac{\pi}{N}} \frac{g(y)}{x+y}\,\d y
		\lesssim \sum_{k=1}^4 \sum_{m=k-3}^k
			\bigl| S_n (f_{k,m}\chn)(x) \bigr|
		\quad\text{for $x\in(\tfrac{\pi}{4N},\infty)$.}
\end{equation*}
Finally, we multiply both sides by $\chn$ and then,
taking the limes inferior as $n\to \infty$,
we get
\begin{equation*}
	\int_0^{\infty} \frac{g(y)}{x+y}\,\d y
		\lesssim \liminf_{n\to\infty}\sum_{k=1}^4 \sum_{m=k-3}^k
			\chn(x)\bigl| S_n (f_{k,m}\chn)(x) \bigr|
			\quad\text{for $x\in\R_+$}
\end{equation*}
where we used that $c_n-T_n^{17}/\sqrt{n}\to 1$ as $n\to\infty$ due to assumption~\eqref{E:chn-def} and property of $c_n$~\eqref{E:cn-def}.
\end{proof}

\begin{lemma} \label{L:Sn-FN}
Let $\rho_Z$ be an \ri norm, $n\in\N$ and let $N$ be as in \eqref{E:N-def}.
Then, for any $f\in\MM(\R)$ and $T>0$,
\begin{equation*} 
	\left\| S_{n} (\fT) - c_n F_N (\fT) \right\|_{Z(-T,T)}
		\lesssim \frac{T^{17}}{\sqrt{n}}\, \| S f^* \|_{Z\dom}.
\end{equation*}
\end{lemma}

\begin{proof}
Let $x\in(-T,T)$.
Relation~\eqref{E:Sneq} yields the equality
\begin{align} \label{E:diff}
    S_{n} (\fT)(x) - c_n F_N(\fT)(x)
        =  c_{n}\sum_{k=2}^7\,\int_{-T}^{T} \frac{K_7^{(n)}(x,y)}{x-y} f(y)\,\d y
            + c_n I(x) + c_n II(x),
\end{align}
where $I$ and $II$ are given by \eqref{E:I-def} and \eqref{E:II-def}, respectively.
Using the estimate~\eqref{E:SE-summary-2} and Lemma~\ref{L:HleS}, we get
\begin{align} \label{E:Sum-K-bound}
	\begin{split}
        \sum_{k=2}^7\,
            \biggl\| \int_{-T}^{T} \frac{K_7^{(n)}(x,y)}{x-y} f(y)\,\d y \biggr\|_{Z(-T,T)}
    		  &\lesssim \sum_{j=1}^6 \widetilde\gamma_j \bigl\| H (m_j \fT) \bigr\|_{Z(-T,T)}
                \\
              &\lesssim \biggl(\sum_{j=1}^6 \widetilde\gamma_j \|m_j\|_{L^\infty(-T,T)}\biggr) 
                \| S f^* \|_{Z\dom}
            \lesssim \frac{T^{17}}{\sqrt{n}}\, \|Sf^*\|_{Z\dom},
        \quad 
	\end{split}
\end{align}
where the last inequality is due to~\eqref{E:sums-gm}.

Next, we estimate the term $I$.
Using trigonometric identities, we infer that with $x\in(-T,T)$,
\begin{align*}
	I(x) & = \cos(Nx) \sin\Bigl(\frac{x}{2N}\Bigr) \int_{-T}^{T}
				\cos(Ny) \cos\Bigl(\frac{y}{2N}\Bigr) \frac{f(y)}{x-y}\,\d y
				\\
		&\quad - \sin(Nx) \sin\Bigl(\frac{x}{2N}\Bigr) \int_{-T}^{T}
				\sin(Ny) \cos\Bigl(\frac{y}{2N}\Bigr) \frac{f(y)}{x-y}\,\d y
				\\
		&\quad - \cos(Nx) \cos\Bigl(\frac{x}{2N}\Bigr) \int_{-T}^{T}
				\cos(Ny) \sin\Bigl(\frac{y}{2N}\Bigr) \frac{f(y)}{x-y}\,\d y
				\\
		&\quad + \sin(Nx) \cos\Bigl(\frac{x}{2N}\Bigr) \int_{-T}^{T}
				\sin(Ny) \sin\Bigl(\frac{y}{2N}\Bigr) \frac{f(y)}{x-y}\,\d y,
\end{align*}
which can be rewritten as
\begin{equation*}
	I(x) = \sum_{k=1}^4 w_k(x) H(u_k \fT)(x),
\end{equation*}
in which
$\|w_k\|_{L^\infty(-T,T)}\le T/2N$ 
and 
$\|u_k\|_{L^\infty(-T,T)}\le 1$
for $k=1,2$.
Also
$\|w_k\|_{L^\infty(-T,T)}\le 1$
and
$\|u_k\|_{L^\infty(-T,T)}\le T/2N$
for $k=3,4$.
Using Lemma~\ref{L:HleS} again, we conclude that
\begin{align}\label{E:I-bound}
    \begin{split}
	\|I\|_{Z(-T,T)}
        & \le \sum_{k=1}^4 \|w_k\|_{L^\infty(-T,T)} \|H(u_k \fT)\|_{Z(-T,T)}
            \\
        & \lesssim \biggl(\sum_{k=1}^4 \|w_k\|_{L^\infty(-T,T)}\|u_k\|_{L^\infty(-T,T)}\biggr)
            \|Sf^*\|_{Z\dom}
          \lesssim \frac{T}{\sqrt{n}} \|Sf^*\|_{Z\dom}.
    \end{split}
\end{align}
The term $II$ can be treated analogously to obtain
\begin{equation} \label{E:II-bound}
	\|II\|_{Z(-T,T)} \lesssim \frac{T^2}{n} \|Sf^*\|_{Z\dom}.
\end{equation}
Combining all the estimates \eqref{E:Sum-K-bound}, \eqref{E:I-bound} and \eqref{E:II-bound} into \eqref{E:diff},
we obtain
\begin{equation*}
     \| S_{n} (\fT) - c_n F_N(\fT) \|_{Z(-T,T)}
        \lesssim \biggl(\frac{T^{17}}{\sqrt{n}} + \frac{T}{\sqrt{n}} + \frac{T^2}{n} \biggr)
            \|Sf^*\|_{Z\dom}
        \lesssim \frac{T^{17}}{\sqrt{n}} \|Sf^*\|_{Z\dom}.
        \qedhere
\end{equation*}
\end{proof}

\section{Density results} \label{S:density}

\noindent
We recall a few notions related to \ri spaces first.
Let $\rho_X$ be an \ri norm. Its fundamental function
$\varphi_X$ is given for $t\ge 0$ by
\begin{equation*}
	\varphi_X(t) = \rho_X(\chi_{E_t}),
\end{equation*}
in which $E_t\subset\RR$ is any measurable set obeying $|E_t|=t$.
The fundamental function is nondecreasing and absolutely continuous except
perhaps at the origin.

A function $f$ in an \ri space $X(\RR)$ is said to have an absolutely continuous norm
if $\|f\|_{X(E_n)}\to 0$ for every sequence $\{E_n\}$ of measurable sets in $\RR$
such that $E_n\downarrow \emptyset$.
The set of all functions in $X(\RR)$ having an absolutely continuous norm is denoted by $X_a(\RR)$.

Any finite sum $\sum a_n \chi_{E_n}$, in which $a_n\in\R$ and $E_n$ is a measurable subset of $\RR$ of a finite measure is called a simple function.
The simple functions belong to any \ri space.  The closure of the
set of simple functions in an \ri space $X(\RR)$ is denoted by $X_b(\RR)$.

A connection between the sets $X_a(\R)$, $X_b(\R)$
and the closure of $\DD(\R)$ is the subject of the next result.
Recall that $\DD(\R)$ denotes the set of infinitely differentiable and
compactly supported functions in $\R$.

\begin{lemma} \label{L:Dclosure}
Let $\rho_X$ be an \ri norm.
The following statements are equivalent:
\begin{enumerate}
\item\label{en:C}
$X_b(\R)$ is the closure of the continuous functions with compact support;
\item\label{en:D}
 $X_b(\R)$ is the closure of $\DD(\R)$;
\item\label{en:lim}
$\lim_{t\to 0+} \varphi_X(t)=0$;
\item\label{en:AB}
$X_a(\R) = X_b(\R)$.
\end{enumerate}
\end{lemma}

\begin{proof}
The equivalence of \ref{en:C} and \ref{en:lim} is proved in \citep[Chapter~3,
Lemma~6.3]{Ben:88} and the equivalence of \ref{en:lim} and \ref{en:AB} is shown
in \citep[Chapter~2, Theorem~5.5]{Ben:88}.  Clearly \ref{en:D} implies
\ref{en:C}.

We now prove that \ref{en:AB} implies \ref{en:D}.
It suffices to show that each
simple function $f$ can be approximated in $X$ by an infinitely differentiable
compactly supported function.
To this end, let $\psi_\delta$ be a mollification family, \ie let $\psi_1$ be a smooth function supported in $(-1,1)$ satisfying $0\le \psi_1\le 1$, $\int \psi_1 =1$ and set $\psi_\delta(t) = \tfrac 1\delta\psi(\delta t)$ for $\delta>0$.
Let us then define $f^\delta = f *\psi_\delta$, the mollification of $f$.
Clearly, $f^\delta$ is smooth, compactly supported and $|f^\delta|\le |f|$.
Furthermore $f^\delta(x)\to f(x)$ for \ae $x\in\R$.  Since $f\in X_a(\R)$, Proposition~3.6 in Chapter~1 of \citep{Ben:88} yields $\|f^\delta-f\|_{X(\RR)}\to 0$ as $\delta\to 0^+$.
\end{proof}

Given $a>0$, define the dilation operator, $D_a$, at $f\in\MM(\RR)$ by
\begin{equation} \label{E:Da-def}
	D_af(t) = f\bigl(\tfrac ta\bigr)
		\quad\text{for $t>0$}.
\end{equation}
The operator $D_a$ is bounded on every \ri space, that is,
\begin{equation} \label{E:Da-XX}
	\|D_a f\|_{X(\RR)} \le C \|f\|_{X(\RR)}
		\quad\text{for every $f\in X(\RR)$},
\end{equation}
where $C\le\max\{1,a\}$; see \citep[Chapter~3, Proposition~5.11]{Ben:88}.
We will make use of this fact in the subsequent proofs.

Let us now introduce an important condition on a given \ri norm $\rho_Y$.
Namely,
\begin{equation} \label{E:logY}
	\eta\in Y\dom,
	\quad\text{where}\quad
		\eta(t) = \log\left( 1+\tfrac1t \right)
		\quad\text{for $t>0$}.
\end{equation}
Observe that $\eta$ may be equivalently replaced
by the function
$(1-\log t)\chi_{(0,1)}(t) + \tfrac{1}{t}\chi_{(1,\infty)}(t)$.
We will later see that condition \eqref{E:logY}
is a generalization of condition \eqref{E:logB} to the class of \ri spaces.

\begin{lemma} \label{L:FN-f}
Let $\rho_Y$ be an \ri norm satisfying \eqref{E:logY}. Then,
\begin{equation} \label{E:FN-f}
	\lim_{N\to\infty} \|F_N f - f\|_{Y(\R)} = 0
	\quad\text{for every $f\in\DD(\R)$}.
\end{equation}
\end{lemma}

\begin{proof}
Let $f$ be given and assume that it vanishes outside of $(-T,T)$ for some $T>0$.
Set $R=2T+1$.
We have
\begin{equation} \label{E:jun1}
	\| F_Nf - f \|_{Y(\R)}
		\le \| F_Nf - f \|_{Y([-R,R])}
			+ \|F_N f\|_{Y(\R\setminus[-R,R])}.
\end{equation}
Now, $F_Nf\to f$ uniformly on $[-R,R]$ as $N\to\infty$ and therefore
the first term of \eqref{E:jun1} goes to zero as $R\to\infty$.
Let us focus on the second term of \eqref{E:jun1}.
Since we take the norm outside of $[-R,R]$,
we assume that $|x|> R$.
Changing variables, then integrating by parts, we get
\begin{align}
\begin{split}
\label{E:dec8}
	F_Nf(x)
		& = \int_{x-T}^{x+T} \sin(Ny) \frac{f(x-y)}{y}\,\d y
			\\
		& = \frac1{N} \int_{x-T}^{x+T}
			\cos(Ny) \frac{\d}{\d y}\biggl(\frac{f(x-y)}{y}\biggr)\d y.
\end{split}
\end{align}
Since $f$ is smooth, there is a constant $C>0$ such that
\begin{equation*}
	\frac{\d}{\d y}\biggl(\frac{f(x-y)}{y}\biggr)
		\le \frac{\|f'\|_{L^\infty(\R)}}{|y|} + \frac{\|f\|_{L^\infty(\R)}}{y^2}
		\le \frac{C}{|y|},
\end{equation*}
where we used the fact that, since $|x|\ge 2T+1$,
$y\in(x-T,x+T)$ implies that $|y|\ge 1$.
Therefore, \eqref{E:dec8} yields
\begin{equation*}
	|F_Nf(x)|
		\le \frac{C}{N} \int_{x-T}^{x+T} \frac{\d y}{|y|}
		= \frac{C}{N} \log\biggl(1+\frac{2T}{|x-T|}\biggr),
  \quad |x|\ge 2T+1.
\end{equation*}
Finally,
\begin{align*}
	\|F_N f\|_{Y(\R\setminus[-R,R])}
	& \le \frac{C}{N}
			\left\|
				\log\left(1+\tfrac{2T}{|x-T|}\right)
			\right\|_{Y(\R\setminus[-R,R])}
		= \frac{2C}{N}
			\left\|
				\log\left(1+\tfrac{2T}{x-T}\right)
			\right\|_{Y(2T+1,\infty)}
		\\
	& = \frac{2C}{N}
			\left\|
				D_{2T}\left(\log\left(1+\tfrac{1}{x}\right)\right)
			\right\|_{Y(T+1,\infty)}
	\le \frac{4TC}{N}
			\left\|
				\log\left(1+\tfrac{1}{x}\right)
			\right\|_{Y\dom},
\end{align*}
where the last inequality is due to the boundedness of the dilation operator
\eqref{E:Da-def} on $Y$.
Since the last norm is finite by assumption \eqref{E:logY},
we conclude that
$\|F_N f\|_{Y(\R\setminus[-R,R])}\to 0$ as $N\to\infty$.
\end{proof}

\begin{lemma} \label{L:Snf-f}
Let $\rho_Y$ be an \ri norm satisfying \eqref{E:logY}
and let $\chn$ be as in \eqref{E:chn-def}.
Then,
\begin{equation*}
	\lim_{n\to\infty} \|\chn S_n f - f\|_{Y(\R)} = 0
	\quad\text{for every $f\in\DD(\R)$}.
\end{equation*}
\end{lemma}

\begin{proof}
Let $f\in\DD(\R)$.
Since it is compactly supported, we may, without loss of generality, assume that its support is contained in $(-T_{1}, T_{1})$.
We have
\begin{align} \label{E:dec9}
	\begin{split}
	\|\chn S_n f - f\|_{Y(\R)}
		\le \|S_n f - c_n F_Nf \|_{Y(-T_n,T_n)}
			+ |c_n-1| \| F_Nf \|_{Y(\R)}
			+ \|F_Nf-f\|_{Y(\R)},
		\end{split}
\end{align}
where $N$ and $c_n$ are defined by \eqref{E:N-def} and \eqref{E:cn-def},
respectively.

An application of Lemma~\ref{L:Sn-FN} shows that
\begin{equation} \label{E:jun2}
	\|S_n f - c_n F_Nf \|_{Y(-T_n,T_n)}
		\lesssim \frac{T_n^{17}}{\sqrt{n}} \|Sf^*\|_{Y\dom}.
\end{equation}
Next, since $f^*$ is bounded and supported in $(0,T_{1})$
\begin{equation} \label{E:dec10}
	Sf^*(t)
		= \int_{0}^{T_{n_1}} \frac{f^*(s)}{s+t}\,\d s
		\le C \log\biggl(1+\frac{T_{1}}{t}\biggr)
\end{equation}
with $C=\|f\|_{L^\infty(\R)}$.
Therefore, as the dilation operator \eqref{E:Da-def} is bounded, 
\begin{equation*}
	\|Sf^*\|_{Y\dom}
		\le C \left\|
						D_{T_{1}}\!\left(\log\left(1+\tfrac{1}{t}\right)\right)
					\right\|_{Y\dom}
		\le CT_{1} \left\|
						\log\left(1+\tfrac{1}{t}\right)
					\right\|_{Y\dom},
\end{equation*}
and consequently, thanks to assumption \eqref{E:logY},
\begin{equation} \label{E:Sf-finite}
	\|Sf^*\|_{Y\dom} < \infty.
\end{equation}
Since $T_n^{17}/\sqrt{n}\to 0$ as $n\to\infty$ due to the choice of $T_n$ in
\eqref{E:chn-def}, inequality \eqref{E:jun2} ensures that the first term of
\eqref{E:dec9} vanishes as $n\to\infty$.

As for the second term of \eqref{E:dec9}, we claim that
\begin{equation} \label{E:FleS}
	\|F_N f\|_{Z(\R)}
		\lesssim \|Sf^*\|_{Z\dom}.
\end{equation}
Indeed, by definition, we have
\begin{align*}
	F_N f(x)
		& = \int_{\R} \frac{\sin\bigl(N(x-y)\bigr)}{x-y} f(y)\,\d y
			\\
		& = \sin(Nx) \int_{\R} \cos(Ny)\frac{f(y)}{x-y}\,\d y
				- \cos(Nx) \int_{\R} \sin(Ny)\frac{f(y)}{x-y}\,\d y
		\quad\text{for $x\in\R$},
\end{align*}
and thus
\begin{equation*}
	\|F_N f\|_{Y(\R)}
		\le \|H(m_1f)\|_{Y(\R)}
			+ \|H(m_2f)\|_{Y(\R)},
\end{equation*}
in which $\|m_j\|_{L^\infty(\R)} = 1$, $j=1,2$ and \eqref{E:FleS} follows by
Lemma~\ref{L:HleS}.
Inequality \eqref{E:FleS}, together with
\eqref{E:Sf-finite}, therefore yields
\begin{equation*}
	|c_n-1| \| F_Nf \|_{Y(\R)}
        \le |c_n-1| \|Sf^*\|_{Y\dom}
        \to 0,
\end{equation*}
since $c_n\to 1$ as $n\to\infty$.

Finally, the third term in~\eqref{E:dec9} tends to zero due to
Lemma~\ref{L:FN-f}.
\end{proof}

\section{Proof of Theorem~\ref{T:ri}}
\label{S:general}

\noindent
Before we present the proof, we need an auxiliary result asserting that the convergence problem reduces to the uniform boundedness of all the partial sum operators.

\begin{lemma} \label{L:UBP}
Let $\rho_X$ and $\rho_Y$ be \ri norms.
Suppose that $\varphi_X(0+)=0$, $X_b(\R) \subseteq Y(\R)$
and that \eqref{E:logY} holds for $Y\dom$.
Then,
\begin{equation}\label{en:converges}
	\lim_{n\to\infty} \|\chn S_n (f\chn) - f\|_{Y(\R)} = 0
		\quad\text{for every $f\in X_b(\R)$};
\end{equation}
if and only if
\begin{equation}\label{en:uniform}
	\sup_{n\in\N}\; \|\chn S_n (f\chn) \|_{Y(\R)} \lesssim \|f\|_{X(\R)}
	\quad\text{for every $f\in X_b(\R)$}.
\end{equation}
\end{lemma}

\begin{proof}
To show that \eqref{en:uniform} implies \eqref{en:converges}, set
\begin{equation*}
	M = \bigl\{f\in X_b(\R): \|\chn S_n (f\chn) -f \|_{Y(\R)} \to 0 \bigr\}.
\end{equation*}
Our aim is to show that $X_b\subseteq M$.
Lemma~\ref{L:Dclosure} asserts that, under the assumption $\varphi_X(0+)=0$, the space $X_b(\R)$ coincides with the closure of $\DD(\R)$ in $X(\R)$, proving that $\overline{\DD(\R)}^{X}\subseteq M$.
By Lemma~\ref{L:Snf-f}, we have $\DD(\R)\subset M$.
Therefore, it only remains to show that $M$ is closed in $X(\R)$.
To this end, suppose that $\{f_k\}$ is a sequence of functions in $M$ converging to $f$ in the norm of $X(\R)$.
Thanks to the embedding $X_b(\R)\subseteq Y(\R)$, we have
\begin{equation} \label{E:jun3}
	\| f - f_k \|_{Y(\R)} \lesssim \| f - f_k \|_{X(\R)}
\end{equation}
for all $k$.
Then, using \eqref{E:jun3} and \eqref{en:uniform},
\begin{align*}
	\|\chn S_n(f\chn) - f \|_{Y(\R)}
		& \le \bigl\| \chn S_n\bigl((f - f_k)\chn\bigr) \bigr\|_{Y(\R)}
			+ \| \chn S_n(f_k\chn) - f_k \|_{Y(\R)}
			+ \| f_k - f \|_{Y(\R)}
		\\
		& \lesssim \| f - f_k \|_{X(\R)}
			+ \| \chn S_n (f_k\chn) - f_k \|_{Y(\R)}.
\end{align*}
Given $\varepsilon>0$ take $k_0$ sufficiently large so that
$\| f_{k_0} - f\|_{X(\R)} < \varepsilon$.
To this $k_0$, we can associate a positive integer $n_0\in\N$ such that $\|\chn S_n(f_{k_0}\chn) - f_{k_0}\|_{Y(\R)}<\varepsilon$ for every $n\ge n_0$.
Therefore, $\|\chn S_n(f\chn) - f\|_{Y(\R)}\lesssim\varepsilon$
for $n\ge n_0$, proving that $f\in M$.

The converse follows immediately by the Uniform Boundedness Principle, since $X_b(\R)$ is a complete space as a closed subspace of the Banach space $X(\R)$.
\end{proof}

The next lemma requires the notion of an associate space.
With \ri norm $\rho$, the functional $\rho'\colon\MM^+(\RR)\to\R$ given by
\begin{equation*}
	\rho'(g) = \sup
		\left\{
			\int_{\RR} fg\,\d x:
				f\in\MM^+(\RR),\,\rho(f)\le 1
		\right\}
\end{equation*}
is also an \ri norm, called the K\"othe dual or the associate norm of $\rho$.
If $X(\RR)$ is the \ri space determined by an \ri norm $\rho_X$, then the \ri space $X'(\RR)$ determined by $\rho_X'$ is called the associate space of $X(\RR)$.
It holds that $(X')'(\RR)=X(\RR)$, see \citep[Chapter~1,Theorem~2.7]{Ben:88}.
Especially,
\begin{equation} \label{E:reflexivity}
	\rho(f) = \sup
		\left\{
			\int_{\RR} fg\,\d x:
				f\in\MM^+(\RR),\,\rho'(g)\le 1
		\right\}.
\end{equation}

\begin{lemma} \label{L:SXb}
Let $\rho_X$ and $\rho_Y$ be \ri norms and let $S$ be the Stieltjes transform.
Then,
\begin{equation} \label{E:Sri}
    \bigl\| S g \bigr\|_{Y\dom} \lesssim \|g\|_{X\dom}
        \quad\text{for every $g\in X\dom$},
\end{equation}
if and only if
\begin{equation} \label{E:Sri'}
	\bigl\| S g \bigr\|_{Y\dom} \lesssim \|g\|_{X\dom}
		\quad\text{for every $g\in X_b\dom$}.
\end{equation}
\end{lemma}

\begin{proof}
\let\dom=\relax
Clearly, \eqref{E:Sri} implies \eqref{E:Sri'}. Conversely, \eqref{E:Sri'} asserts
that there is $C>0$ such that
\begin{equation} \label{E:jun6}
	\sup\left\{ \|Sg\|_{Y\dom}:
				g\in X_b\dom,\, \|g\|_{X\dom}\le 1
			\right\}
		\le C.
\end{equation}
Note that all the spaces in this paragraph are taken over $\R_+$.
Next, by~\eqref{E:reflexivity},
\begin{equation} \label{E:jun7}
	\|Sg\|_{Y\dom}
		= \sup\left\{ \int_{0}^\infty Sg(x)f(x)\,\d x :
						f\in Y',\, \|f\|_{Y'} \le 1
					\right\},
\end{equation}
where $Y'$ denotes the associate space of $Y$.
Clearly,
\begin{equation} \label{E:S-selfadj}
	\int_{0}^\infty Sg(x)f(x)\,\d x
		= \int_{0}^\infty g(x)Sf(x)\,\d x
\end{equation}
due to Fubini's theorem.
Using \eqref{E:S-selfadj} and \eqref{E:jun7}, \eqref{E:jun6} is equivalent
to
\begin{equation} \label{E:jun8}
	\sup\left\{ 
				\int_{0}^\infty g(x)Sf(x)\,\d x:
					g\in X_b,\,
					f\in Y',\,
					\|g\|_{X}\le 1,\,
					\|f\|_{Y'}\le 1
			\right\}
		\le C.
\end{equation}
The result of \citep[Chapter~1, Theorem~3.12]{Ben:88}
asserts that the supremum over $X_b$ coincides with the one over $X$, that is
\begin{equation*}
	\sup\left\{ 
				\int_{0}^\infty g(x)h(x)\,\d x:
					g\in X_b,\,
					\|g\|_{X}\le 1,\,
			\right\}
	=
	\sup\left\{ 
				\int_{0}^\infty g(x)h(x)\,\d x:
					g\in X,\,
					\|g\|_{X}\le 1,\,
			\right\}
\end{equation*}
for any eligible $h$. In conclusion, we may replace $X_b$ by $X$ in inequality
\eqref{E:jun8}. We use \eqref{E:S-selfadj} and \eqref{E:jun7} again, to arrive back
to \eqref{E:jun6} with $X_b$ replaced by $X$ which is nothing but \eqref{E:Sri}.
\end{proof}

\begin{proof}[Proof of Theorem~\ref{T:ri}]
Lemma~\ref{L:UBP} asserts that \ref{en:convergence-ri} is equivalent to
the following statement.
\begin{enumerate}[label={\rm(\roman*)'}]
    \item\label{en:uniform-ri}
    The norm $\rho_Y$ obeys \eqref{E:logY} and
    \begin{equation*}
	   \sup_{n\in\N}\; \|\chn S_n (f\chn) \|_{Y(\R)} \lesssim \|f\|_{X(\R)}
	   \quad\text{for every $f\in X_b(\R)$}.
    \end{equation*}
\end{enumerate}
Next, thanks to Lemma~\ref{L:SXb}, assertion \ref{en:S-ri} holds if and only if
\begin{enumerate}[label={\rm(\roman*)'}]
    \setcounter{enumi}{1}
    \item\label{en:S-ri'}
    \begin{equation*}
    	\bigl\| S g \bigr\|_{Y\dom} \lesssim \|g\|_{X\dom}
    		\quad\text{for every $g\in X_b\dom$}.
    \end{equation*}
\end{enumerate}

Now, it is, enough to prove the equivalence of \ref{en:uniform-ri} and \ref{en:S-ri'}.
Let us prove that \ref{en:uniform-ri} implies \ref{en:S-ri'} first.
Let $g\in X_b\dom$. We may assume that $g$ is nonnegative. Using Lemma~\ref{L:SleSn}
and the lower semi-continuity of \ri norms, we obtain
\begin{equation} \label{E:jun4}
	\| S g \|_{Y\dom}
		\lesssim \liminf_{n\to\infty}\sum_{k=1}^4 \sum_{m=0}^3
			\bigl\| \chn S_n (f_{k,m}\chn) \bigr\|_{Y(\R)},
\end{equation}
where the functions $f_{k,m}$ are given by \eqref{E:fkm-def}. By their definition,
\begin{equation} \label{E:jun5}
	\|f_{k,m}\|_{X(\R)} \le \|g\|_{X\dom}
		\quad\text{for all eligible $k$, $m$ and $n$.}
\end{equation}
Also $f_{k,m}\in X_b(\R)$. Therefore, we may apply the uniform bound from
\ref{en:uniform-ri}, and \eqref{E:jun4} with \eqref{E:jun5} yield
\begin{equation*}
	\| S g \|_{Y\dom}
		\lesssim \sup_{n\in\N}\sum_{k=1}^4 \sum_{m=0}^3
			\|f_{k,m}\|_{X(\R)}
		\lesssim \|g\|_{X\dom},
\end{equation*}
proving \ref{en:S-ri'}.

Conversely, assume that \ref{en:S-ri'} holds and let $f\in X_b(\R)$.
Lemma~\ref{L:SnleS} implies that
\begin{equation*}
	\|\chn S_n (f\chn) \|_{Y(\R)}
		\lesssim \biggl(1+\frac{T_n^{17}}{\sqrt{n}}\biggr) \|Sf^*\|_{Y\dom}
		\lesssim \|f^*\|_{X\dom}
		= \|f\|_{X(\R)},
\end{equation*}
where we used \ref{en:S-ri'} with $g=f^*\in X_b\dom$ together with the fact that $T_n^{17}/\sqrt{n}$ is bounded, by the choice of the sequence $T_n$ in~\eqref{E:chn-def}.
This proves the inequality in \ref{en:uniform-ri}.
It remains to show that $\rho_Y$ satisfies \eqref{E:logY}.
Let us set $g=\chi_{(0,1)}$.
Then $g\in X_b\dom$ and
\begin{equation*}
	Sg(t) = \int_{0}^{1} \frac{\d s}{s+t} = \log\left( 1+\tfrac1t \right)
		\quad\text{for $t\in\R_+$}
\end{equation*}
and \eqref{E:logY} follows from~\ref{en:S-ri'}, since $g\in X_b\dom$.
\end{proof}

\section{Orlicz spaces and proof of Theorem~\ref{T:main-Orlicz}}
\label{S:Orlicz}

\noindent
We first recall definitions and basic properties of Orlicz spaces, see \citep[Chapter~4, Section~8]{Ben:88} or \citep{Rao:91} for further reference.
Let $A$ be a Young function and let $\RR$ denote $\R$ or $\R_+$.
The Orlicz space $L^A(\RR)$ is defined by
\begin{equation*}
	L^A(\RR) =
	\left\{
		f\in\MM(\RR):
			\int_{\RR} A(c|f|) \le 1 \text{ for some $c>0$}
	\right\}
\end{equation*}
and equipped with Luxemburg norm
\begin{equation*}
	\|f\|_{L^A(\RR)}
		= \rho_A(|f|)
		= \inf
		\left\{
			\lambda>0:
				\int_{\RR} A(|f|/\lambda) \le 1
		\right\}.
\end{equation*}
The functional $\rho_A$ is an \ri Banach function norm.
Its fundamental function, denoted by $\varphi_A$, is given by
\begin{equation} \label{E:Orl-fundamental}
	\varphi_A(t) = \frac{1}{A^{-1}(1/t)}
		\quad\text{for $t>0$},
\end{equation}
where $A^{-1}$ stands for the generalised right-continuous inverse of $A$.
For Young functions $A$ and $B$, one has that $L^A(\RR)\subseteq L^B(\RR)$
if and only if there is $c>0$ such that $B(t)\le A(ct)$ for $t\ge 0$.
The space $E^A(\RR)$, defined by
\begin{equation*}
	E^A(\RR) =
	\left\{
		f\in\MM(\RR):
			\int_{\RR} A(c|f|)\,\d\mu <\infty \text{ for all $c>0$}
	\right\},
\end{equation*}
coincides with the subspace of functions in $L^A(\RR)$ having an
absolutely continuous norm that is
\begin{equation} \label{E:Orl-AC}
	L^A_a(\RR) = E^A(\RR);
\end{equation}
see \citep[Theorem~4.12.13]{Pic:13}.
With $f_n, f\in L^A(\RR)$,
$n\in\N$, one has
\begin{equation} \label{E:norm-convergence}
	\lim_{n\to\infty} \|f - f_n\|_{L^A(\RR)} = 0
\end{equation}
if and only if
\begin{equation} \label{E:modular-convergence}
	\lim_{n\to\infty} \int_{\RR} A\bigl(\lambda|f-f_n|\bigr)\,\d\mu = 0
		\quad\text{for all $\lambda>0$.}
\end{equation}
The proof that \eqref{E:norm-convergence} implies \eqref{E:modular-convergence} is given in~\citep[Section~3.4, Theorem~12]{Rao:91}.
Moreover, as remarked in \citep[Section~3.4, Remark, p.~87]{Rao:91},
Morse and Transue have shown the converse in their paper \citep{Mor:50}, where it appears as an observation necessary to a rather involved argument.

\begin{proof}[Proof of Theorem~\ref{T:main-Orlicz}]
Set $\rho_X=\rho_A$ and $\rho_Y=\rho_B$. Let us verify the
hypotheses of Theorem~\ref{T:ri}. By \eqref{E:Orl-fundamental},
$\varphi_A(s)=1/A^{-1}(1/s)$ for $s>0$. Therefore
$\varphi(0+)=0$ if and only if $A^{-1}(t)\to\infty$ as $t\to\infty$.
This happens if and only if $A$ is finite-valued.
Next, $L^A(\R)\subseteq L^B(\R)$ if and only if $B(t)\le A(ct)$
for $t\ge 0$ with global constant $c>0$. Finally, $\rho_B$
obeys \eqref{E:logY} if and only if \eqref{E:logB} holds, by the
very definition of the Luxemburg norm.

Theorem~\ref{T:ri} now asserts that \eqref{E:logB} holds and
\begin{equation} \label{E:Sn-LB}
	\lim_{n\to\infty} \|\chn S_n(f\chn) - f \|_{L^B(\R)}
		= 0
	\quad\text{for every $f\in L^A_b(\R)$}
\end{equation}
if and only if
\begin{equation} \label{E:S-LA-LB}
	\|Sg\|_{L^B\dom}
		\lesssim \|g\|_{L^A\dom}
	\quad\text{for every $f\in L^A\dom$}.
\end{equation}
First, \eqref{E:Sn-LB} is equivalent to \ref{en:truncated-convergence} since
Lemma~\ref{L:Dclosure} together with \eqref{E:Orl-AC} ensures that
\begin{equation} \label{E:D-Orl}
	L^A_b(\R)
	=
	\overline{\DD(\R)}^{L^A}
	=
	L^A_a(\R)
	=
	E^A(\R),
\end{equation}
and the rest is due to the equivalence of \eqref{E:norm-convergence} and \eqref{E:modular-convergence}.

It only remains to show that \eqref{E:S-LA-LB} is equivalent to \ref{en:AB-conditions}.
For $g\in\MM^+\dom$, define
\begin{equation*}
	Pg(t) = \frac{1}{t}\int_{0}^t g(s)\,\d s
	\quad\text{and}\quad
	Qg(t) = \int_{t}^\infty \frac{g(s)}{s}\,\d s
	\quad\text{for $t\in\R_+$}.
\end{equation*}
Then
\begin{equation*}
	Pg(t) + Qg(t)
		= \int_{0}^\infty \min\left\{ \tfrac{1}{s},\tfrac{1}{t} \right\}g(s)\,\d s
\end{equation*}
and, since
\begin{equation*}
	\frac{1}{s+t}
		\le \min\left\{ \tfrac{1}{s},\tfrac{1}{t} \right\}
		\le \frac{2}{s+t}
		\quad\text{for $s,t\in\R_+$},
\end{equation*}
we infer that $Sg\le Pg + Qg \le 2Sg$. Consequently, as $P$ and $Q$ are positive
operators, \eqref{E:S-LA-LB} holds if and only if for every $g\in L^A\dom$
\begin{equation} \label{E:P-LA-LB}
	\|Pg\|_{L^B\dom}
		\lesssim \|g\|_{L^A\dom}
\end{equation}
and
\begin{equation} \label{E:Q-LA-LB}
	\|Qg\|_{L^B\dom}
		\lesssim \|g\|_{L^A\dom}.
\end{equation}
Finally, the fact that~\eqref{E:P-LA-LB}
and \eqref{E:Q-LA-LB} are characterised by the first and
the second inequality of~\ref{en:AB-conditions}, respectively,
is well-known; see~\citep{Cia:14a,Cia:99a,Gal:88,Kit:97,Ker:19}, for instance.
\end{proof}

\paragraph{Acknowledgment}
We would like to extend our special thanks to Ron Kerman.
This project was initiated in 2015 by V.~Musil and R.~Kerman and, following a hiatus, was independently revisited by S.~Spektor and R.~Kerman.
The final manuscript is the result of a renewed collaboration between all three researchers in 2024--2025.
Although Ron Kerman chose not to be included as a co-author, his contribution to this project is indisputable, and we gratefully acknowledge his essential input.
We also wish to thank the referee for their careful reading of the paper and for their valuable comments.

\end{document}